\documentclass[11pt,a4paper, envcountsame]{amsart}
\usepackage[usenames,dvipsnames]{color}

 \usepackage{amsmath, amssymb, xspace}
 \usepackage{graphicx}

 \usepackage[all]{xy}
\usepackage{arydshln}

\usepackage{tikz}
\usetikzlibrary{arrows,snakes,positioning,backgrounds,shadows}

\newtheorem{theorem}{Theorem}[section]

\newtheorem{fact}[theorem]{Fact}

\newtheorem{remark}[theorem]{Remark}

\newtheorem{prop}[theorem]{Proposition}

\newtheorem{claim}[theorem]{Claim}

\theoremstyle{definition}
\newtheorem{definition}[theorem]{Definition}

\newtheorem{example}[theorem]{Example}

\newcommand{\NN}{{\mathbb{N}}}

\newcommand{\ZZ}{{\mathbb{Z}}}
\newcommand{\sub}{\subseteq}
\newcommand{\sN}[1]{_{#1\in \NN}}

\newcommand{\PI}[1]{\Pi^0_{#1}}

\newcommand{\bi}{\begin{itemize}}
\newcommand{\ei}{\end{itemize}}
\newcommand{\bc}{\begin{center}}
\newcommand{\ec}{\end{center}}

\newcommand{\tp}[1]{2^{#1}}
\newcommand{\ex}{\exists}

\newcommand{\la}{\langle}
\newcommand{\ra}{\rangle}

\newcommand{\n}{\noindent}

\newcommand{\sss}{\sigma}
\newcommand{\aaa}{\alpha}

\newcommand \seq[1]{{\left\langle{#1}\right\rangle}}

\newcommand\+[1]{\mathcal{#1}}

\newcommand{\wt}{\widetilde}
\newcommand{\ol}{\overline}
\newcommand{\ul}{\underline}

\DeclareMathOperator{\UT}{\text{\it UT}}

\begin{document}

\title{Word automatic  groups of nilpotency class~2}

  \author{Andr\'e Nies and Frank Stephan}

\maketitle

\begin{abstract}  We consider word automaticity for groups that are nilpotent of class $2$ and have exponent a prime $p$. We show that the infinitely generated free group in this variety is not word automatic. In contrast, the infinite extra-special $p$-group $E_p$ is word automatic, as well as an intermediate group $H_p$ which has an infinite centre. In the last section we introduce a method to show automaticity of central extensions of abelian groups via co-cycles. \end{abstract}

\section{Introduction}
A   structure  in a finite signature is called \emph{word automatic}  (or FA-presen\-table)  if the domain is a set of strings that can be recognized by a finite automata (FA) over an alphabet~$\Sigma$. The atomic relations can be recognized by FA as well, as follows: 
 To check  whether an atomic relation holds for elements $a_1, \ldots, a_n$ of the domain, the strings   representing these elements are extended to strings $\sss_1, \ldots,   \sss_n$  of the same length by  using a filler symbol $\Diamond$ that is not in $\Sigma$; one requires that a finite  automaton 
 recognizes the language over $(\Sigma \cup \{\Diamond\})^n$ that consists of  the  strings obtained by stacking such strings $\sss_1, \ldots, \sss_n$ on top of each other.

  Word automatic structures were  first considered by Hodgson~\cite{Hodgson:76,Hodgson:83}  who used them  to  give a new proof that the first-order theory of $(\NN, +)$ is decidable. They were  studied in depth   by  Khoussainov and Nerode~\cite{Khoussainov.Nerode:94},       the  founding paper of the area. 
 

Since finite automata  are    devices of a very limited computational power, it can be  challenging to find nontrivial examples of word automatic structures in a particular class; in some cases, such as for the class of Boolean algebras, it can be shown that only the obvious structures are word automatic~\cite{Khoussainov.etal:07}. 

  The  class of  word automatic groups (not to be confused with automatic groups in the sense of Thurston) has good closure properties; for instance, it is closed under finite direct products,  and quotients by regular normal subgroups. In     the abelian setting is also closed under a certain type    of FA-recognizable amalgamation~\cite{Nies.Semukhin:09}. We note that  for any   finite group $S$, the direct power $S^{(\omega)}$   is   word automatic. 
 
  \medskip
  
\n \emph{The abelian case.}     
Nies  and Semukhin~\cite{Nies.Semukhin:09}  constructed word automatic  torsion-free indecomposable abelian groups of rank  $ 2$ and larger. This sets them apart from examples such as $\ZZ \times \ZZ$, which can be considered word automatic in a trivial way.
Their examples involved divisibility by more than one prime. It is still open whether a group of this kind can be obtained as a subgroup of $(\mathbb Z[1/p])^n$ for any $n \ge 2$.   Braun and Str\"ungmann~\cite{Braun.Struengmann:11}, building on methods of Tsankov~\cite{Tsankov:11}, provided strong restrictions on torsion free abelian groups. In particular,   such groups  have finite rank.

 \medskip

	  \n \emph{New examples of word automatic groups.}  	In this paper we study indecomposable word automatic   groups in  varieties just beyond the abelian.   Recall that a group  $G$ is \emph{nilpotent  of class 2}   if  the law $[[x,y],z]$ holds in~$G$. In other words,  $G$ is a central extension of an abelian group by another.     Fix an odd prime $p$, and    	let  $\+ N_{2,p}$ denote  the variety of groups that are nilpotent of   class~$2$  and have   exponent~$p$ (that is, $x^p= 1$ for each $x$).    Note that for each   $G \in \+ N_{2,p}$, the centre $Z(G)$ is    an elementary abelian $p$-group, and   can thus be seen as a vector space over the field $R$.  The same holds for the central quotient $G/Z(G)$.    
	
	 Let $R=GF(p)$ be the field of $p$ elements. The new  examples of word automatic groups are infinitely generated variants of the
unitriangular group
\bc  {$ \UT_3(R)= \left \{ \left(  \begin{matrix}
   1 & a &  c \\
   0 &  1 &  b \\
   0 &   0 &  1   \end{matrix} \right)  \colon \, a, b,c \in R \right \}$}. \ec

   This is  the free  group of rank 2 in $\+ N_{2,p}$. It is   generated by  $x_0=  \left(  \begin{matrix}
   1 & 1 & 0 \\
   0 &  1 &  0 \\
   0 &   0 &  1   \end{matrix} \right) $ and $ x_1=  \left(  \begin{matrix}
   1 & 0 &  0 \\
   0 &  1 &  1 \\
   0 &   0 &  1   \end{matrix} \right)$.  We have $[x_0,x_1]= z:=  \left(  \begin{matrix}
   1 & 0 &  1 \\
   0 &  1 &  0 \\
   0 &   0 &  1   \end{matrix} \right)$.  
   The centre of $ \UT_3(R)$ is the cyclic group generated by~$z$.  

 We introduce groups $G \in \+ N_{2,p}$ by  varying the definition of   $\UT_3(R)$. We   posit that $G$ have    an  infinite sequence  of distinct generators $  \seq {x_n} \sN n$.   
 If   the  commutators $[x_i,x_k]  $,  $i<k$ are linearly independent,   then $G$ is   the free group  of infinite rank in $\+ N_{2,p}$, which we will denote by $F_\infty$. We   will show that this group  is not word automatic. In a nutshell, for any FA-presentation of $F_\infty$,  the linear independence of the commutators of generators would require         $p^{\binom n 2}$  strings of length $O(n)$ to represent the linear combinations of the $[x_i,x_k]$ for $i < k < n$, which is contradictory for large enough $n$.

In contrast,  if the commutators  $[x_i,x_k] \in Z(G)$ are dependent  in a certain strong way, then $G$  is word automatic. The simplest example is   a  group that we will denote by $E_p$, or just $E$ if $p$ is understood: one requires that there is an element $z \neq e$ such that $[x_i,x_k] = z$ for each $i<k$.  Thus, the centre is  cyclic  as in the case of $\UT_3(R)$. We note that  $E_p$ is an extra-special $p$-group in the sense of Higman and Hall: the centre is cyclic of order $p$, equals the derived subgroup, and  quotient by the centre  is an elementary abelian $p$-group (i.e., a vector space over $GF(p)$). A slightly more complex example is  the group $H_p$:   we require that $[x_i,x_k] = z_k$ for $i<k$, where  the~$z_k$ form  a basis of the centre.  

\medskip

\n \emph{Abelian subgroups of finite index.}   Nies and Thomas~\cite{Nies.Thomas:08} proved  that each finitely generated subgroup of a word automatic group  has an abelian subgroup of finite index. This   indicates that word automatic groups are close to abelian.     It remains open whether each      torsion-free word automatic  group has an abelian subgroup of finite index \cite[Question~4.5]{Nies:DescribeGroups}.   The groups $E_p$ and $H_p$ are 
 word automatic \emph{torsion} groups  without   abelian subgroups of finite index. The first example of  a group of this kind was  \cite[Example 12]{Nies.Thomas:08},  which we will revisit as Example~\ref{ex: Thomas} below.   

%

  \medskip 
\n \emph{The isomorphism problem.}   The  computational complexity of the  isomorphism  problem is a a good indicator of the complexity of a class of word automatic structures. This problem asks whether two   presentations given by finite automata describe   isomorphic structures. It was a central topic in Khoussainov, Nies, Rubin and Stephan~\cite{Khoussainov.etal:07}. They showed that  the isomorphism problem is  $\Sigma^1_1$-complete  for word automatic graphs, but decidable for word automatic Boolean algebras. Kuske, Liu and Lohrey~\cite{Kuske.etal:13} proved  that  the isomorphism problem for word automatic equivalence relations  is $\PI 1$-complete. For abelian groups, as well as    groups at large, the  complexity is unknown (\cite[Question 5.1]{LogicBlog:21}). To show that the isomorphism problem for a class of word automatic structures  is undecidable,  one   attempts   to find constructions for sufficiently complicated structures in the class in order to encode some undecidable problem. In contrast, towards showing  its  decidability,  one attempts to  provide restrictions on   structures in the class. 
It is unknown at present whether   the isomorphism problem for the class of  word automatic   groups in $\+ N_{2,p}$ decidable.    The present paper follows both approaches to this,  by providing both examples and restrictions for this class.
 
   \medskip
   
\n \emph{Some model theory of   extra-special $p$-groups.} We note that the extra-special $p$-groups $E_p$ have appeared in numerous places in the literature. For instance, let us 	briefly review some model theoretic properties of these groups  for $p\neq 2$,  due to Felgner~\cite{Felgner:75}. (He denotes these groups by  $G(p, \le)$ where $\le$ is the usual  ordering of $\omega$.)
 On page   423 he  provides a recursive  axiom system for the theory of $E_p$.  
It  expresses that the group has  exponent $p$, and that the centre is cyclic of order $p$ and contains the derived subgroup, which  is non-trivial. Furthermore, it expresses that the  quotient by the centre is infinite,  using an infinite list of axioms. This implies that $E_p$ is $\omega$-categorical, 
 since up to isomorphism there is only one countably infinite extra-special group of exponent $p$ as shown in Newman~\cite{Newman:60}.

  A group  is called pseudofinite if every  first-order sentence that holds in it also  holds in a finite group (for background on this notion see~\cite{Houcine.Point:13}).   Note that for each odd $k$ there is an extra-special  group of exponent $p$ and   order $p^k$.  
  So, any  finite set of the axioms can be satisfied in a finite model. Hence $E_p$ is pseudo-finite.

   \section{Preliminaries on word automaticity and on groups in $\+ N_{2,p}$}
\label{s:prelim}

\n \emph{Presentations via finite automata: facts  and examples}   \begin{definition} \label{def:FA} One says that 
 a   structure $A$ in a finite signature is   \emph{word automatic} (or   FA  presentable) if  the elements of the domain can be represented, possibly ambiguously,  by the strings in a regular language $D_A$ over an alphabet $\Sigma$ such that the following holds. For each  atomic  relation of  the type $Ra_1, \ldots, a_n$, or $f(a_1, \ldots, a_n)= b$, or $a=b$,  where $R,f$ are $n$-ary relation, respectively  function symbols and $a_1, \ldots, a_n, b \in A$, there is a     finite automaton that recognizes it in the sense discussed at the beginning of the paper. An \emph{FA presentation} is a collection of FA as above. \end{definition}

\begin{example} The structure  $(\NN,+)$ is word automatic via  the usual binary expansion of a natural number.  The  {alphabet} is $\Sigma= \{0,1\}$, and the  { domain}  consists of the  strings ending in  $1$,  and the empty string. 
One  represents a number $n$ by a binary string  $\aaa$, with the \emph{least} significant digit first.  Thus $n = n_\aaa = \sum_{i< |\aaa|} \aaa(i) \tp  i$.
The empty string denotes $0$.
  A finite automaton over the alphabet $(\Sigma \cup \{\Diamond\})^3$  checks the correctness of  the sum  $n_\aaa + n_\beta = n_\gamma$ via the  carry bit procedure, where the carry bit moves to   the right.  
For instance,   if $\aaa = 101, \beta = 11001$ and $\gamma =00011$, the automaton checks that the sum is correct   by accepting the string
\[ \begin{array}{|c|c|c|c|c|}\hline 
   1  & 0 & 1  & \Diamond& \Diamond   \\
     1 & 1 & 0 &  0  &  1\\
   0 &   0 &  0 &  1 &  1 \\  \hline \end{array} . \]
   \end{example} 

 In the  definition of FA presentations above,  note that   for terms $s,t$ the equation $s=t$ is an atomic relation. So,   this definition allows for  equality to  be a nontrivial equivalence relation $E$ on $D_A$, which   on occasion is useful in defining FA presentations. However, since the length-lexicographical ordering $\le_L$ on $D_A$ is FA-recognizable, one can   replace $D_A$ by the regular set $\{x \in D_A \colon \forall y \, [ y E x \, \to \,  x \le_L y]\}$, and thereby  uniquely represent  elements by strings.

  Given an  FA presentation of a  structure $\mathcal A$ and a formula
$\phi$  (possibly with parameters), one can effectively determine a
 finite automaton recognizing the relation  on $\mathcal A$ defined by $\phi $.
    The  proof is by  induction on $|\phi|$.  To deal with  existential quantifiers, one uses that
for each non-deterministic finite automaton, there is a deterministic one that recognizes the same language.   Hence,  to show that a group $(G, \circ, {}^{-1})$ is word automatic, it suffices to provide an FA recognizing  the binary group operation $\circ$; the unary group operation is definable from it.

Automatic groups in the sense of Epstein and Thurston (see ~\cite{Epstein:92})  are finitely generated by definition.
 In contrast, the appropriate  setting  for obtaining  interesting \emph{word} automatic groups is outside the  finitely generated, for  a finitely generated  group is word automatic if and only if  it is  has an abelian subgroup of  finite index.  Note that  word automatic groups are called finite automata presentable in~\cite{Nies:DescribeGroups}  in order to avoid confounding the two notions.

 \medskip 
 
\n \emph{Preliminaries on nilpotent-2 groups.} 
 If a sequence $\seq{x_i}\sN i$ of   generators of a group  $G$ under discussion is fixed, for   a tuple of integers~ $\aaa$ we let   \bc $\ul {\alpha} := \prod_{i< |\aaa|} x_i^{\alpha_i}$. \ec
For  group elements $x,y$,   we define the commutator by $[x,y]= x^{-1}y^{-1}xy$.  If $G$ is nilpotent of class 2,   the commutator induces an alternating bilinear form \bc $\theta_G \colon ( G/Z)^2  \to Z$ \ec where $Z= Z(G)$ and  $\theta_G(aZ, bZ) =[a,b]$ which is well-defined.
Hence, for $|\aaa| = |\beta|=n$ and central elements $c,d$,
\[ \big [ \ul {\alpha} c, \ul {\beta} d \big ]= \prod_{r<s\le n} [x_r, x_s]^{\aaa_r \beta_s - \aaa_s \beta_r}\]
This identity will  be   used below without  mention.

\medskip


\section{The group  $F_\infty$  is not  word automatic}
 Recall that by $F_\infty$ we denote the free group of infinite rank in the variety  $\+ N_{2,p}$ of groups of exponent $p$ and nilpotency class 2.  Let $\seq{x_i}\sN i$ be a sequence of free generators, and recall that we write
$ {\ul \alpha} := \prod_{i< |\aaa|} x_i^{\alpha_i}$. It is well known (\cite{Warfield:06}
that each element of $F_\infty$ has a    normal form 
\begin{equation} \label{free normal form} \ul {\alpha} \cdot  \prod_{r< s < n } [x_r,x_s]^{v_{r,s}},  \end{equation}
where $|\aaa| = n$ and $0\le v_{r,s} <p$. (To make this normal form unique, one can   require that   $\aaa_{n-1} \neq 0 \vee \ex r [v_{r,n-1}\neq 0]$; i.e., $n$ is chosen minimal.  The central elements are the ones where $\aaa$ is a string of $0$s.)

\begin{theorem}  $F_\infty$ is not word automatic.\end{theorem}
\begin{proof} 
Assume for a contradiction that $F_\infty$ has an FA presentation.  It      has 
as a    domain a regular  set  $D\sub \Sigma^*$ for some alphabet $\Sigma$, and has an FA-recognizable operation $\circ \colon D \times D \to D$ for the binary group operation.  As discussed in Section~\ref{s:prelim}, we may assume that the equivalence relation   denoting equality in the represented group is   equality on $D$.  Hence we can identify $F_\infty$ with the structure on the  domain $D$. 
As before,  by $\le_L$ we denote the length-lexicographical ordering on~$D$, which  is FA-recognizable as well.

  The    idea of the proof  is to     define sequences     $\seq{u_n}\sN n$ and $\seq{z_n}\sN n$ of elements of~$D$ such that the central elements $[u_i, u_k]$, $i<k<n$,  are linearly independent over the field $GF(p)$, and for each $n$, the subgroup generated by  $u_0, \ldots, u_n$ is contained in the set of strings  $\{x \in D \colon x  \le_L z_{n+1}\}$. Furthermore,    $z_{n+1}$ is obtained from $z_n$ by  applying a   function $G$  that is first-order definable in the structure on $D$ enriched by $\le_L$.  So this function is FA-recognizable. The pumping lemma  now implies  that the length of  any linear combination of the $[u_i, u_k]$,   where $i< k < n$,  is $O(n)$. This means that for large $n$ there are not enough  strings of the allowed length to accommodate  these $p^{\binom n 2}$  linear combinations.

For the detail, we verify three claims. Note that  the centre $Z= Z(F_\infty) $ is a regular set.   The first claim states that for each finite set $S\sub D$ there is $u \in D$ such that the map $ZS/Z \to Z$ given by $Zr \to [r,u]$ has range disjoint from $S$,  and is injective.  
\begin{claim} \label{cl: uS} For each finite set $S \sub D$, there is a   string $u \in D$  such that 
\bi
\item[(i)] if  $r \in S\setminus Z $ then $    [r,u]\not\in S$;
\item[(ii)] for each $  v,w \in S$ such that   $ [v,u]= [w,u] $, there is $  c \in Z $ such that $  c\circ v = w$.
\ei
Given $S$, we will write $u(S)$ for the $\le_L$-least   string $u$ satisfying the claim.
\end{claim}
\n To see this, let $u =x_k$ where $k $ is so large that only $x_i$ with $i< k$ occur in the normal form of any element of $S$. 

For (i) note that $r$ contains some $x_i$ with $i<k$, so the   normal form of $[r,u]$ contains $[x_i,x_k]$, while the normal form of an element of $S$ does not contain such commutators.

 For (ii) let $v= \ul {\alpha} d$ with $d$ central. Then the normal form of $[v,u]$ ends in $\prod_{i<k} [x_i,x_k]^{\aaa_i}$, which  determines $\aaa$.  This verifies the claim.

 \medskip
 
 We now recursively  define  the sequences $\seq{u_n}\sN n$  and $\seq{z_n}\sN n$   in $D$. 
 Let~$z_0$ be the string representing the neutral element~$1$. 
  Suppose now that  $z_n$ has been defined.  Let \bc $u_n = u(V_n)$  where  $V_n=\{v \colon v \le_L z_n\}$, \ec according to Claim~\ref{cl: uS}. Next, let $z_{n+1}$ be the $\le_L$-least string $z$ such that 
 \begin{equation} \label{eqn:zn+1} \text{for each } v,w \in V_n  \text{ and each }  {i<p}, \text{ one has }w \circ u_n^i \circ [v,u_n] \le_L z.\end{equation}
 In the   structure $(D, \circ, \le_L)$, one can define  $u_n$ from~$z_n$  in a uniform first order way. Hence,    there is a   function $G\colon D \to D $ that  is first-order definable in this  structure such  that $z_{n+1} = G(z_n)$ for each $n$. As mentioned in the introduction, its  definability  implies that its  graph of    $G$ can be recognized by a finite automaton. In particular,    we have $|z_n| = O(n)$ since by the  pumping lemma, the output of a function with regular graph is by  at most   a constant longer than the input.

In the following, we will write $u_{i,k}$ for $[u_i, u_k]$ in case that  $i\neq k$. 
\begin{claim} \label{cl:subgroup}  $\la u_0 , \ldots, u_n \ra\sub V_{n+1}$ for each $n$. \end{claim}
\n We use  induction on $n$. For $n=0$ we have $\la u_0 \ra \sub V_1$ because    in  (\ref{eqn:zn+1}) we can let $v,w$ be strings denoting the neutral element. For the inductive step, note that by the normal form (and freeness of $F_\infty$) each element of $\la u_0 , \ldots, u_n \ra$ has the form 
\bc $ y=\prod_{i \le n} u_i^{\aaa_i}   \prod_{r< s \le n } [u_{r,s}]^{\gamma_{r,s}}$ \ec
where $0\le\aaa_i,  \gamma_{r,s} <p$.
This can be rewritten as $wu_n^{\aaa_n} [v,u_n]$ where
 \bc $w= \prod_{i <  n}   u_i^{\aaa_i} \prod_{r<s < n}   [u_{r,s}]^{\gamma_{r,s}}$
 and 
 $v= \prod_{k<n} u_k^{\gamma_{k,n}}$. \ec
 By inductive hypothesis $w,v \in V_n$. So the element $y$ is in $V_{n+1}$ by  (\ref{eqn:zn+1}).
This verifies the claim.

In the next claim we view elementary abelian $p$-groups as vector spaces over the field $GF(p)$.
\begin{claim} \label{cl:indep}  \mbox{} \bi \item[(a)] The   elements $Z u_i$ are linearly independent in  $G/Z$. 
\item[(b)] The elements $u_{i,k}$ are   linearly independent in  $Z$.
\ei
\end{claim}
In both (a) and (b) we use induction over  an upper  bound $n$ on  the indices. Both statements  hold  vacuously  for $n=0$.
For (a) note that $u_{n+1}^i \not \in V_{n+1} Z$ for each $i<p$: Otherwise there is   $c\in Z $ such that $v:= u_{n+1}^i c \in V_{n+1}$. We have  $u_{n+1}^i \not \in  Z$ and hence $v\not \in Z$. Since  $[v,u_{n+1}]=1$ and $1= z_0\in V_{n+1}$, this contradicts condition  (i) of Claim~\ref{cl: uS} for $S= V_{n+1}$. Therefore, by the   Claim~\ref{cl:subgroup},  $Z\la u_{n+1}\ra  \cap  Z \la u_0, \ldots, u_n \ra =0$. 

For (b), inductively the $u_{i,k}$,  $i<k \le n$ form  a basis for  a subspace $T_0\sub  Z$.  The linear map $Zw\to  [w,u_{n+1}]$ defined on $ZV_n/Z$ is injective  by (ii) of Claim~\ref{cl: uS}. So, by (a), the $[u_i,u_{n+1}]$ for $i \le n$  form a basis of   a subspace $T_1\sub  Z$. Then $T_0 \cap T_1 =\{0\}$: if   $\sum_{i \le n}{\aaa_i}[u_i, u_{n+1}] = [\sum _{i \le n} {\aaa_i}u_i, u_{n+1}]\neq 0$ for some coefficients $\aaa_i$ with $0 \le \aaa_i< p$, then $r:= \sum_{i \le n} {\aaa_i}u_i \in V_{n+1} \setminus Z$ by    Claim~\ref{cl:subgroup}, and hence $[r, u_{n+1}] \not \in V_{n+1}$ by condition (i) of Claim~\ref{cl: uS}.  By Claim~\ref{cl:subgroup}  again this implies $[r, u_{n+1}]  \not \in T_0$.   This concludes the inductive step and verifies  the claim.

  Given $n$, by Claim~\ref{cl:subgroup}  we have \bc  $\prod_{i<k < n} [u_i, u_k]^{\gamma_{i,k}} \in V_n$  \ec for each array $\seq{\gamma_{i,k}}_{i<k < n}$ of exponents in $[0,p)$. By Claim~\ref{cl:indep}  all these elements are distinct. Since $V_n$ consists of strings of length $O(n)$,  we have $p^{n(n-1)/2}$ distinct strings of length $O(n)$, which is contradictory  for large enough~$n$.
\end{proof}
  
%
%

 \section{Quotients of $F_\infty$ with dependency between the commutators}

    As before,  let  $F_\infty$ denote the   group in the variety  $\+ N_{2,p}$  with   free generators $x_i$, $i \in \NN$. We next define     groups $E, H $ in $\+ N_{2,p}$ as quotients of $F_\infty$. 
For   $E$ we require that  all the  commutators $[x_i,x_\ell]$, $i< \ell$, be equal.  
For   $H$  we require   that  $[x_i,x_k]= z_k$ for $i< k$ where the $z_k$ are linearly independent over $\ZZ/p\ZZ$.  
Formally, we define the groups via presentations in the variety $\+ N_{2,p}$:
\begin{eqnarray}  E &=&   \la \{x_i\} \sN i,z  \colon  [x_i,x_k] =z  \  (0\le i<k) \ra   \label{f1} \\
 H &= &  \la  \{x_i\} \sN i,\{z_k\}_{ k \in \NN^+}  \colon   [x_i,x_k] =z_k  \   (i<k) \ra \label{f2} \end{eqnarray}
 Note that $E$   is  extra-special  as discussed in the introduction. 
In contrast, the centre of  $H$  has infinite dimension.

%


In a nilpotent group, each nontrivial normal subgroup intersects the centre non-trivially. This implies that every proper quotient of $E$ is abelian; in particular, $E$ is not residually finite. On the other hand,    $H$ is residually a finite $p$-group, and hence is canonically embedded   into its pro-$p$ completion. To see this, take an element $h \neq 1$ that can be written in terms of the  generators $x_0, \ldots, x_{n-1}$. Then $h \neq 1$ in the finite  $p$-group which is the   quotient of $H$ by the normal subgroup generated by the $x_k$, $k \ge n$ (which contains all the $z_k$, $k \ge n$). 

We supply two algebraic facts supporting the claim that the FA presentations of the groups we  provide in the next section will be nontrivial: neither can they be obtained from FA presentations of abelian groups, nor are they combinations of  FA presentations of simpler components.
\begin{prop} None   of the groups $  E, H$ has an abelian subgroup of finite index. \end{prop}
\begin{proof} It suffices to show this for $E$, because    $H$ has  $E$ as a quotient. So suppose $M$ is a  subgroup of $E$  with  finite index.   There are $k< r <s$ such that $x_r x_k^{-1}\in M$ and $x_s x_k^{-1}\in M$. We have  
$[x_r x_k^{-1} , x_s x_k^{-1}]= [x_r, x_s][x_r, x_k^{-1}][x_k^{-1},x_s]= z$, so these two elements of $M$ don't commute. \end{proof}
\begin{prop}  The groups $E$ and $H$ are  indecomposable. 
\end{prop}
\begin{proof}  
$E$ is indecomposable because each of its proper quotients is abelian.   Suppose that $H= A \times B$ for subgroups $A,B$.  Write $Z= Z(H)$, $\wt x_i = x_i Z$, $\wt A = AZ/Z$, and $\wt B= BZ/Z$. For a string $\beta$ over $\ZZ$ we write $\wt \beta = \prod_{i< |\beta|} \wt x_i ^{\beta_i} $. If  $A \sub Z$ then all $\wt x_i $ are in $\wt B$, so $z_k= [x_0,x_k]$ is  in $B$ for each $k \ge 1$, and hence  $B=H$.   By symmetry we may assume that $A, B \not \sub Z$.  

First suppose that $ \wt x_0^m \in \wt A$ for some $m$ with $0<m< p$. Then there is a string  $\beta$  of length at least $2$ such that $\wt \beta \in \wt B$. Hence $[x_0^m, \ul \beta]= \prod_{0< i < |\beta|}z_i^{m\beta_i}\neq e$, contradiction. 

Now assume otherwise. There are strings $\aaa, \beta$ over $\{0, \ldots, p-1\}$ with $\wt \aaa \in \wt A, \wt \beta \in \wt B$ such that $\wt x_0= \wt \aaa + \wt \beta$. By   assumption $\aaa_k \neq 0$ for some   $k>0$, chosen to  be least. Then $\beta_0= 1-\aaa_0$, $\beta_k = - \aaa_k$ and $\beta_r= 0$ for $0< r< k$. Hence,  using that the commutator is bilinear,   $[\ul \aaa, \ul \beta]$ contains a factor $z_k^\gamma$ where $\gamma = \aaa_k \beta_0 - \aaa_0 \beta_k = \aaa_k$. Thus $[\ul \aaa, \ul \beta]\neq e$, contradiction. 
%
%
%
%
%
\end{proof}

\section{The groups $E$ and $H$ are word automatic}  \label{s:EH}
  Recall that, fixing a prime $p$, the groups $E$ and $H$ are defined via   (\ref{f1}) and (\ref{f2}).  We will describe    FA presentations of these groups based on the  alphabet \bc $\Sigma=\{0, \ldots, p-1\}$. \ec 
The variables $\aaa, \beta, \gamma$ will denote strings over $\Sigma$. Recall that we write  $\ul {\alpha} $ for $ \prod_i x_i^{\alpha_i}$, where the $x_i$ are the generators given by the presentations.
 Using the normal form for elements in $F_\infty$ given in  (\ref{free normal form}),   each element of $E$ can be written in the  form 
  $ z^v \cdot \ul {\alpha} $ 
where   $0\le v <p$.   

 In this section all arithmetic is modulo~$p$.  Strings  $\alpha, \beta$  in expressions such as $\aaa + \beta$ are   thought of as extended by $0$s if necessary, and  will be  added component-wise.

In the case of the group $E$, one notes  that for $n= \max(|\aaa|, |\beta|)$,  \begin{equation} \label{eqn: check}  \ul {\alpha} \cdot \ul {\beta} =z^{-\sum_{k=1}^{n-1}\alpha_k (\sum_{i=0}^{k-1} \beta_i)} \cdot \ul{\alpha + \beta}.\end{equation}
  This is because   one can  calculate $\ul {\alpha} \cdot \ul {\beta}  = \prod_k x_k^{\alpha_k} \prod_i x_i^{\beta_i}$ by, for  decreasing positive~$k$,    moving    terms $x_k^{\alpha_k}$ to the right past the terms $x_i^{\beta_i}$ for $i=0, \ldots, k-1$, and then joining it with $x_k^{\beta_k}$ to form  $x_k^{\aaa_k + \beta_k}$. Each such move creates a factor $z^{-\alpha_k \beta_i}$ in the centre of $E$.

 \begin{prop} The group $E$ is word automatic. \end{prop}

\begin{proof}   
An element   $  z^v \cdot \ul {\alpha}  $  is represented by the string  $v \alpha$.  The domain consists of the strings $v\aaa$ such that $\aaa$ is empty, or its last entry is not $0$.

We describe an FA that checks (\ref{eqn: check}), and hence correctly verifies the binary group operation. It  processes an   input composed of a triple of strings $v\alpha, w\beta, r\gamma $ stacked on three tracks.      If necessary, we extend the    strings     by~$0$s to   make them have length $1+n$ where $n = \max(|\aaa|, |\beta|, |\gamma|)$.
So the FA processes an input over the alphabet $\Sigma^3$ in this format:
\begin{center}
$s_{-1} s_0\ldots s_{n-1} = \begin{array}{|l|l|l|l|l|l|l}
\hline
v & \aaa_0 & \aaa_1 & \cdots & \aaa_{n-1} \\
\hline 
w & \beta_0  & \beta_1 & \cdots &  \beta{n-1} \\
\hline
r& \gamma_0  & \gamma_1 & \cdots &  \gamma_{n-1}  \\

\hline

\end{array}$
\end{center}
When the  FA scans  the first stack symbol     $s_{-1} $,    it stores it in the constant size memory given by  the state. As it scans $s_k$ for $k=0, \ldots, n-1$  it   checks  that  $\aaa_k + \beta_k = \gamma_k$ (if this fails,    it enters a rejecting state and remains in it until the whole  input has been  scanned). The FA  stores  the current value $\sum_{i< k} \beta_i$ in its constant size internal memory. It    adds   $\alpha_k (\sum_{i< k} \beta_i)$ to a variable $x$ ranging over~$\Sigma$, with initial value $0$,  also thought of as   stored     in the internal memory. After scanning the last symbol it checks whether $v+w-x =r \mod p $, and accepts accordingly.     \end{proof}

Each element of $H$  can be written in the form 
\bc $\prod_{s < |\aaa| }  z_s^{v_{s}}  \cdot \ul {\alpha}  $ \ec
where  $0\le v_{s} <p$. The central elements are the ones where $\aaa$ consists only of $0$s.  
If  $n= \max(|\aaa|, |\beta|)$,  similar to (\ref{eqn: check}) we have in $H$ that 
 \begin{equation} \label{eqn:check2} \ul {\alpha} \cdot \ul {\beta} = \prod_{k=1^{n-1}} z_k^{-\alpha_k (\sum_{i< k} \beta_i)}  \cdot  \ul {\alpha + \beta}.  \end{equation}
 
\begin{prop} The group  $H$ is word automatic. \end{prop}
\begin{proof}  An element $h$ of the group $H$ will be  represented by a pair of strings  $\alpha,   v$ over $\Sigma$, of the same length $n$, such that if $n>0$ then  $v_0= 0$,     not both end in $0$,    and  $h=  \prod_{k=1}^n z_k^{v_k} \cdot \ul {\alpha} $. These strings are written on two tracks, with  $\alpha$ on top of~$  v$.

  We describe an FA that checks  (\ref{eqn:check2}), and thus  correctly verifies the group operation.
 It processes  strings over  the  alphabet $\Sigma^6$ of the  format
\begin{center}
$s_0\ldots s_{n-1}= \begin{array}{|l|l|l|l|l|l|l}
\hline
  \aaa_0 & \aaa_1 & \cdots & \aaa_{n-1} \\
  
  \hdashline 
    v_0 & v_1 & \cdots & v_{n-1} \\
\hline 
  \beta_0  & \beta_1 & \cdots &  \beta{n-1} \\
    \hdashline 
    w_0 & w_1 & \cdots & w_{n-1} \\
\hline
  \gamma_0  & \gamma_1 & \cdots &  \gamma_{n-1}  \\
    \hdashline 
    r_0 & r_1 & \cdots & r_{n-1} \\
\hline

\end{array}.$
\end{center}
  At the beginning,  the FA scans $s_0$ and checks that $v_0= w_0= r_0 = 0$. As in the case of $E$, when the FA scans $s_k$ for $k \ge 0$,     it   checks  that  $\aaa_k + \beta_k = \gamma_k$, and   stores  the current value $\sum_{i< k} \beta_i$ in its constant size internal memory. However, now it also  checks whether \bc $v_k+w_k-\alpha_k (\sum_{i< k} \beta_i) =r_k\mod p$  \ec   (which holds trivially if $k=0$).  If any of these checks fail it enters a rejecting state. Otherwise,  when the whole input is scanned it accepts.
\end{proof}

B\"uchi automata are  (nondeterministic) finite automata that  work on infinite words. Such a word is accepted if some computation processing it is infinitely often in an accepting state. B\"uchi automatic structures were first considered by Hodgson~\cite{Hodgson:83}, who called them macro-automatic. For background  see \cite[Section 2.1]{Nies:DescribeGroups}.
 We  sketch an example of a  non-abelian uncountable group that    is B\"uchi automatic in a nontrivial way. 
\begin{fact} The pro-$p$ completion   $\hat H_p$ of $H_p$ is B\"uchi automatic. \end{fact}

\begin{proof} For each pair of  infinite words  $\aaa$, $v$  over $\{0, \ldots, p-1\}$, in  $\hat H_p$  one can form the limits $\prod_i x_i^{\aaa_i} \colon = \lim_n  \prod_{i<n} x_i^{\aaa_i}$ and  $\prod_{k> 0}  z_k^{v_k}$.  An element~$h$ of the group $\hat H_p$ is represented by a pair of {infinite} words  $\alpha, v$ such that $h=  \prod_{k>0}  z_k^{v_{k-1}} \cdot \prod_i x_i^{\alpha_i}$.  
%
 We can use the same automaton as above, now   working on infinite words, to verify  the  binary group operation on $\hat H_p$.  \end{proof}
 
 \section{Constructing  word automatic groups via cocycles} 
 
We  review some well-known facts on central extensions of abelian groups. Given two abelian groups $A$ and $Q$, a   \emph{central extension} of $Q$ by $A$
is an   exact sequence 
$0 \to A \to L \to Q \to 0$ such that  $A \sub Z(L)$.    Cocycles are used to describe such extensions.     A~\emph{cocycle} is a  function $f \colon Q \times Q \to A$  such that    
\[ f(u,v)+ f(u+v, w)=f(v,w) + f(u,v+w).\] 
On $Q \times A$, the operation 
\[(u,a) \cdot  (v,b) : = (u+ v, a+ b + f(u,v))\]
defines a group $L_f$, which is abelian iff $f$ is symmetric.  It is easy to verify that with the maps $a \to (0,a)$ and $(u,a) \to u$, one obtains an exact sequence $0 \to A \to L_f \to Q \to 0$. 
   The inverse of $(u,a)$ is $(-u.-a-f(u,-u))$. For associativity,   if we calculate $[(u,a) + (v.b)]+ (w,c)$, the ``correcting term" in the second component on the right side is $f(u,v) + f(u+v,w)$. If we calculate $(u,a) + [(v.b)+ (w,c)]$, the correcting term is $f(v,w)+ f(u,v+w)$.

   Conversely, given an   exact sequence $0 \rightarrow A \rightarrow L \rightarrow Q \rightarrow 0$ with $A \le Z(L)$,   to determine a 2-cocycle $c \colon Q \times Q \rightarrow A$ that yields an equivalent extension, 
one picks a set $T$ of coset representatives  for $A$ in $L$.  
 Fixing a bijection $Q \to T$  and writing $\ol q$ for the image of $q$,
the  cocycle is given by \bc $f(q_0,q_1) \colon = \ol {q_0 \phantom +}  \cdot  \ol {q_1 \phantom + } \cdot  (\ol {q_0+ q_1}) ^{-1} \in A$.\ec   
For detailed background  see e.g.\  Fuchs~\cite[Ch.\ 9]{Fuchs:15} (who calls these objects   ``extensions of $A$ by $Q$",  but uses the same order, first  $Q$ then $A$,  in the  notation).

\begin{prop}  \label{thm:cocyc}  Let $A$ and $ Q$ be word automatic  abelian groups.  Let $L $ be a central  extension of $Q$ by $A$, given by an exact sequence $0 \rightarrow A \rightarrow L \rightarrow Q \rightarrow 0$ with $A \le Z(L)$.    Suppose some   cocycle $f \colon Q \times Q \rightarrow A$ describing this extension is FA recognizable.   Then $L$ is word automatic.   \end{prop}  
\begin{proof} $L$ can be constructed as the set $Q \times A$ with the operation  given above.
 So $L$  can be
interpreted in a first-order way  in the word automatic two-sorted  structure   $ (Q \sqcup A, +_Q, +_A, f)$. This shows that $L$ is word automatic. \end{proof}

Thus, if we can choose $T$ so that the corresponding cocycle $f$ can be computed by a finite automaton, we   obtain an FA-presentation for $L$.    
As an example of how to apply Prop.~\ref{thm:cocyc}, we revisit a   group, introduced in   \cite[Example 12]{Nies.Thomas:08},  that does not have an abelian subgroup of finite index. We give a short proof based on cocycles that this group  is  word automatic.   
 \begin{example} \label{ex: Thomas} Let the group $L$ have  generators $x, y_i, z_k$ ($i,k \in \NN$) subject to the relations
\bc  $y_i^2 = z_k^2 =1  \qquad 
 [y_i,z_k]= [y_i, x]= 1 \qquad 
   z_i^{-1}x z_i = xy_i$   \ec
 We have an  exact sequence $0 \rightarrow A \rightarrow L \rightarrow Q \rightarrow 0$ 
where   $A= Z(L) =  \la x^2, \{y_i\}\sN i \ra$,  and
$Q:= L/A  $ is the direct sum of groups $Ax$ and $Az_i$.  Write $v= Ax$ and $w_i= Az_i$,   so that $v^2 =w_i^2= 1$ in $Q$.
Elements  of $Q$ have  a normal form $q_{s, \aaa}=v^s \cdot \prod_i w_i^{\aaa_i}$ where $s=0,1$ and $\aaa$ is a  bit string (thought to be extended by $0$s if necessary).   
 
The transversal $T$ consists of the elements      $\ol {q_{s, \aaa}}= x^s   \prod_i z_i^{\aaa_i} \in L$.      
     Using that $z_i x= xz_i y_i$, one verifies that the corresponding  2-cocycle is \bc $c(q_{s, \aaa},q_{t, \beta})= \ol {q_{s, \aaa} \phantom +}  \cdot  \ol {q_{t, \beta} \phantom + } \cdot  (\ol {q_{s, \aaa}+ q_{t, \beta}}{ \, }) ^{-1}= x^{2s+t} \prod_i y_i^{\gamma_i} $  \ec where $\gamma_i = (s+2t) \aaa_i + (s+t) \beta_i \mod 2$.  It is clear that this cocycle can be computed by an FA.     \end{example}

\begin{remark}  {\rm Nies and Semukhin~\cite[Thm.\ 4.2]{Nies.Semukhin:09} showed that  an abelian group that has a word automatic normal subgroup    of finite index of  is in itself word automatic. By Prop~\ref{thm:cocyc} this actually holds without the restriction to abelian groups.} \end{remark}

\begin{remark} {\rm In Section~\ref{s:EH} we showed that   the groups $E_p$ and $H_p$ are word automatic. The automata in the proofs  can also be somewhat simplified using   cocycles.  The elements of the form $\ul \aaa$ form a transversal for the extension. The cocycles  are then given in (\ref{eqn: check}) and (\ref{eqn:check2}), respectively.  For $H$, say, the cocycle maps pairs $\wt \aaa, \wt \beta$ of elements of $Q = H/Z(H)$, where $|\aaa| = |\beta| =n$,  to elements $\prod_{k=1}^{n-1} z_k^{r_k}\in Z(H)$. So the  automaton  verifying the cocycle processes strings of the format 

\begin{center}
$t_0\ldots t_{n-1}= \begin{array}{|l|l|l|l|l|l|l}
\hline
  \aaa_0 & \aaa_1 & \cdots & \aaa_{n-1} \\

\hline 
  \beta_0  & \beta_1 & \cdots &  \beta{n-1} \\
   
\hline
  
    0 & r_1 & \cdots & r_{n-1} \\
\hline

\end{array}.$
\end{center}
As it scans the symbols $t_0, \ldots, t_k, \ldots$, it stores  the current value $\sum_{i< k} \beta_i$ in its constant size internal memory. It   checks whether   $ -\alpha_k (\sum_{i< k} \beta_i) =r_k\mod p$; else it enters a rejecting state.   Otherwise, when the whole input has been scanned, it accepts.
}
\end{remark}

%

%
%
\def\cprime{$'$} \def\cprime{$'$}

%

\end{document}